\documentclass{gtmon_a}
\pdfoutput=1

\usepackage{amscd}
\usepackage{graphicx}
\input xy
\xyoption{all}

\proceedingstitle{Proceedings of the Nishida Fest (Kinosaki 2003)}
\conferencestart{28 July 2003}
\conferenceend{8 August 2003}
\conferencename{International Conference in Homotopy Theory}
\conferencelocation{Kinosaki, Japan}

\editor{Matthew Ando}
\givenname{Matthew}
\surname{Ando}

\editor{Norihiko Minami}
\givenname{Norihiko}
\surname{Minami}

\editor{Jack Morava}
\givenname{Jack}
\surname{Morava}

\editor{W Stephen Wilson}
\givenname{W Stephen}
\surname{Wilson}

\title{Homotopy idempotents on manifolds and Bass' conjectures}

\author{A\,J Berrick}
\givenname{A\,J}
\surname{Berrick}
\address{Department of Mathematics\\
National University of Singapore\\\newline
Kent Ridge 117543\\
Singapore}
\email{berrick@math.nus.edu.sg}
\urladdr{}

\author{I Chatterji}
\givenname{I}
\surname{Chatterji}
\address{Department of Mathematics\\
The Ohio State University\\\newline
231 W 18th Ave\\
Columbus OH 43210\\
USA}
\email{indira@math.ohio-state.edu}
\urladdr{}

\author{G Mislin}
\givenname{G}
\surname{Mislin}
\address{Department of Mathematics\\
ETH Z\"{u}rich\\\newline
8092 Z\"{u}rich\\
Switzerland}
\email{mislin@math.ethz.ch}
\urladdr{}

\volumenumber{10}
\issuenumber{}
\publicationyear{2007}
\papernumber{2}
\startpage{41}
\endpage{62}

\doi{}
\MR{}
\Zbl{}

\arxivreference{}


\received{4 October 2004}
\revised{2 April 2005}
\accepted{}
\published{29 January 2007}
\publishedonline{29 January 2007}
\proposed{}
\seconded{}
\corresponding{}
\version{}

\makeatletter
\def\cnewtheorem#1[#2]#3{\newtheorem{#1}{#3}[section]
\expandafter\let\csname c@#1\endcsname\c@Theorem}

\let\xysavmatrix\xymatrix
\def\xymatrix{\disablesubscriptcorrection\xysavmatrix}
\AtBeginDocument{\let\bar\wbar\let\tilde\widetilde\let\hat\what\let\wtilde\widetilde}

\newtheorem{thm}{Theorem}
\newtheorem{con}{Conjecture}

\newtheorem{Theorem}{Theorem}[section]
\cnewtheorem{Conjecture}[Theorem]{Conjecture}
\cnewtheorem{Corollary}[Theorem]{Corollary}
\cnewtheorem{Proposition}[Theorem]{Proposition}
\cnewtheorem{Lemma}[Theorem]{Lemma}
\cnewtheorem{Remark}[Theorem]{Remark}
\cnewtheorem{Definitions}[Theorem]{Definitions}
\cnewtheorem{Remarks}[Theorem]{Remarks}
\cnewtheorem{Examples}[Theorem]{Examples}
\theoremstyle{definition}
\cnewtheorem{rem}[Theorem]{Remark}
\cnewtheorem{rems}[Theorem]{Remarks}
\cnewtheorem{defi}[Theorem]{Definition}
\cnewtheorem{exm}[Theorem]{Example}
\cnewtheorem{Definition}[Theorem]{Definition}
\cnewtheorem{Example}[Theorem]{Example}

\makeatother  

\makeautorefname{Proposition}{Proposition}
\makeautorefname{Lemma}{Lemma}
\makeautorefname{con}{Conjecture}
\makeautorefname{Theorem}{Theorem}

\nonstopmode
\newcounter{fig}

\renewcommand{\theequation}{\thesection-\arabic{equation}}

\newcommand{\Hei}{{H}^{(2)}_i}

\begin{document}

\begin{htmlabstract}
The Bass trace conjectures are placed in the setting of homotopy
idempotent selfmaps of manifolds. For the strong conjecture,
this is achieved via a formulation of Geoghegan. The weaker form
of the conjecture is reformulated as a comparison of ordinary and
L<sup>2</sup>&ndash;Lefschetz numbers.
\end{htmlabstract}

\begin{abstract}
The Bass trace conjectures are placed in the setting
of homotopy idempotent 
selfmaps of manifolds. For the strong conjecture, this is achieved via a
formulation of Geoghegan. The weaker form of the conjecture is reformulated as
a comparison of ordinary and $L^{2}$--Lefschetz numbers.
\end{abstract}

\maketitle

\subsection*{Preface}

This note has its origins in talks discussing Bass' trace conjecture. After
one such lecture (by IC), R Geoghegan kindly mentioned his geometric
perspective on the matter. Then, when another of us (AJB) spoke about the
conjecture at the Kinosaki conference, he thought that a topological audience
might like to hear about that geometric aspect. Thus, it seemed desirable to
attempt to put the conjecture (and its weaker version) in a setting that would
be as motivating as possible to topologists. The result of that attempt
appears below.

\subsection*{Acknowledgement}
The authors warmly thank both R Geoghegan and B\,J Jiang for
their interest in this work. Second author partially supported
by the Swiss NSF grant PA002-101406 and USA NSF grant DMS 0405032.

\section{Introduction}
\label{sec:intro}

In 1976, H Bass \cite{B} conjectured that for any discrete group $G$,
the Hattori--Stallings trace of a finitely generated projective module
over the integral group ring of $G$ should be supported on the identity
component only.  Despite numerous advances (see, for example, Eckmann
\cite{E}, Emmanouil \cite{Emm}, and our earlier paper \cite{BCM}), this
conjecture remains open in general. In \cite{Geo:LNM}, R Geoghegan gave
the first topological interpretation, in terms of Nielsen numbers (stated
as \fullref{Geoghegan Theorem 4.1'} below). In the setting of selfmaps
on manifolds, this translates to the following.

\begin{thm}
\label{principal}The following are equivalent.

\begin{itemize}
\item [\rm(a)]The Bass conjecture is a theorem.

\item[\rm(b)] Every homotopy idempotent selfmap of a closed, smooth and oriented
manifold of dimension greater than $2$ is homotopic to one that has precisely
one fixed point.
\end{itemize}
\end{thm}

Throughout this paper we use the word ``closed'' to refer to a connected,
compact manifold without boundary. Background material on homotopy
idempotents and related invariants will be discussed in Sections \ref{hi} and \ref{invar CW}. A weaker version of Bass' conjecture
amounts to saying that for any group $G$, the coefficients of the non-identity
components of the Hattori--Stallings trace of a finitely generated projective
module over the integral group ring of $G$ should sum to zero (and not
necessarily be individually zero).

\begin{thm}
\label{weakelyprincipal}The following are equivalent.

\begin{enumerate}
\item[\rm(a)]The weak Bass conjecture is a theorem.

\item[\rm(b)]Every pointed homotopy idempotent selfmap of a closed, smooth and
oriented manifold inducing the identity on the fundamental group, has its
Lefschetz number equal to the $L^{2}$--Lefschetz number of the induced map on
its universal cover.
\end{enumerate}
\end{thm}

Background material regarding $L^{2}$--Lefschetz numbers is explained in \fullref{Lef}. The implication $\mathrm{(a)} \Rightarrow \mathrm{(b)}$ had already been observed by Eckmann in \cite{EN} in a slightly different form. The proofs of these two theorems proceed as follows. \fullref{principal} is derived from the analogous statement for finite
CW--complexes (involving Nielsen numbers) from Geoghegan's work,
which is explained in \fullref{invar CW}.
The transition from CW--complexes to manifolds is done in \fullref{CW->M}. The proof of \fullref{principal} as well as some applications is discussed in \fullref{prfThm1}. For \fullref{weakelyprincipal} the strategy is
somewhat similar: namely, we first prove the statements for
finitely presented groups instead of arbitrary groups and for
finite complexes instead of manifolds (see \fullref{Proof 2}).
To deduce Bass'
conjectures (weak and classical) for arbitrary groups we use a
remark due to Bass (\fullref{GeoBass}).

\section{Review of Bass' conjectures}
\label{Rewiew}

We briefly recall Bass' conjectures. Let $\mathbb{Z}G$ denote the
integral group ring of a group $G$. The \emph{augmentation trace} is
the $\mathbb{Z}$--linear map
$$\epsilon\co\mathbb{Z}G\rightarrow\mathbb{Z},\quad g\mapsto1$$
induced by the trivial group homomorphism on $G$. Writing $[\mathbb{Z}
G,\mathbb{Z}G]$ for the additive subgroup of $\mathbb{Z}G$  generated by
the elements $gh-hg$ ($g,h\in G$), we identify $\mathbb{Z}G/[\mathbb{Z}
G,\mathbb{Z}G]$ (the Hochschild homology group $HH_{0}(\mathbb{Z}G) $)
with $\bigoplus_{[s]\in\lbrack G]}\mathbb{Z}\cdot\lbrack s]$, where
$[G]$ is the set of conjugacy classes $[s]$ of elements $s$ of $G$. The
\emph{Hattori--Stallings trace} of $M=\sum_{g\in G}m_{g}g\in\mathbb{Z}G$
is then defined by
\begin{align*}
\mathrm{HS}(M)  &  =M+[\mathbb{Z}G,\mathbb{Z}G]\\
&  =\sum_{[s]\in\lbrack G]}\epsilon_{s}(M)[s]
\ \in\bigoplus_{\lbrack s]\in\lbrack G]}\mathbb{Z}\cdot\lbrack s],
\end{align*}
where for $[s]\in\lbrack G]$, $\epsilon_{s}(M)=\sum_{g\in\lbrack
s]} m_{g}$ is a partial augmentation. In particular, the component
$\epsilon_{e}$ of the identity element $e\in G$ in the Hattori--Stallings
trace is known as the \emph{Kaplansky trace}
$$\kappa\co\mathbb{Z}G\rightarrow\mathbb{Z},
\quad\sum m_{g}g\mapsto m_{e}.$$
Now, an element of $K_{0}(\mathbb{Z}G)$ is represented by a difference of
finitely generated projective $\mathbb{Z}G$--modules, each of which is
determined by an idempotent matrix having entries in $\mathbb{Z}G$. Combining
the usual trace map to $\mathbb{Z}G$ of such a matrix with any of the above
traces on $\mathbb{Z}G$ turns out to induce a well-defined trace map on
$K_{0}(\mathbb{Z}G)$ that is given the same name and notation as before.
Moreover, $\mathrm{HS}$ and $\epsilon$ are natural with respect to all
group homomorphisms (and $\kappa$ with respect to group monomorphisms). In the
case of a free module $\mathbb{Z}G^{n}$, $\epsilon$ takes the value $n$ and
so is just the rank of the module.

In \cite{B}, Bass conjectured the following.

\begin{con}[Classical Bass conjecture]
\label{Bass}
For any group $G$, the induced map
\[
\mathrm{HS}\co K_{0}(\mathbb{Z}G)\rightarrow\bigoplus_{\lbrack s]\in\lbrack
G]}\mathbb{Z}\cdot\lbrack s]
\]
has image in $\mathbb{Z}\cdot\lbrack e]$.
\end{con}

\begin{con}[Weak Bass conjecture]
\label{WeakBass}
For any group $G$, the induced maps
\[
\epsilon,\kappa\co K_{0}(\mathbb{Z}G)\rightarrow\mathbb{Z}%
\]
coincide.
\end{con}

To clarify the discussion below, it is helpful to consider also
reduced $K$--groups. The inclusion $\left\langle e\right\rangle
\hookrightarrow G$ induces a natural homomorphism
\[
\mathbb{Z}=K_{0}(\mathbb{Z})=K_{0}(\mathbb{Z}\left\langle e\right\rangle
)\longrightarrow K_{0}(\mathbb{Z}G)
\]
whose cokernel is the \emph{reduced }$K$\emph{--group}
$\wtilde{K}_{0}(\mathbb{Z}G)$, equipped with natural epimorphism
$\eta\co K_{0}(\mathbb{Z}G)\twoheadrightarrow\wtilde{K}_{0}(\mathbb{Z}G)$.

\section{Homotopy idempotent selfmaps}
\label{hi}

Let $X$ be a connected CW--complex. A selfmap $f\co X\rightarrow X$ is
called \emph{homotopy idempotent} if $f^{2}=f\circ f$ is
\textsl{freely} homotopic to $f$. Since $X$ is path-connected we
can always assume that $f$ fixes a basepoint $x_{0}\in X$, so that
$f$ induces a (not necessarily idempotent) map
$f_{\sharp}\co \pi_{1}(X)\rightarrow\pi_{1}(X)$. Given a homotopy
idempotent selfmap $f\co X\rightarrow X$ on a \textsl{finite
dimensional CW--complex $X$}, according to Hastings and Heller
\cite{HH} there is a CW--complex $Y$ and maps $u\co X\rightarrow
Y$ and $v\co Y\rightarrow X$ such that the following diagram is
(freely) homotopy commutative:
\begin{equation}
\xymatrix{ X\ar[rr]^f\ar[dr]_{u}\ar@(ur,ul)[rrrr]^f& &X\ar[rr]^f\ar[dr]_{u}& &X\\ &Y\ar
[rr]^{\mathrm{id}}\ar[ur]_{v} & &Y\ar[ur]_{v}.& }\label{Heller}%
\end{equation}
In fact, in this diagram we can arrange that the outside triangles strictly
commute. By replacing the maps by homotopic ones, we can (and do) choose the
maps to preserve basepoints. We then get the following commutative diagram of
groups:%
\[
\xymatrix{
\pi_1(X)\ar[rr]^{f_\sharp}\ar[dr]_{u_{\sharp}}
& &\pi_1(X)\ar[rr]^{f_{\sharp}%
}\ar[dr]_{u_{\sharp}}&  &\pi_1(X)\\
&\pi_1(Y)\ar[rr]\ar[ur]_{v_{\sharp}} & &\pi_1(Y)\ar[ur]_{v_{\sharp}}.&
}%
\]
Here the bottom arrow consists of conjugation by the class of a loop at the
basepoint of $Y$. Looking at the middle triangle, we see that $v_{\sharp}$ is
an injective homomorphism while $u_{\sharp}$ is surjective; hence we can make
the identification
\[
\pi_{1}(Y)\cong v_{\sharp}(\pi_{1}(Y))=\mathrm{Im}(f_{\sharp})\leq\pi_{1}(X).
\]
If the homotopy idempotent $f$ is a pointed homotopy idempotent (meaning that
$f^{2}$ is pointed homotopic to $f$), then $u\circ v\co Y\rightarrow Y$
will induce the identity on $\pi_{1}(Y)$. If we require that $f_{\sharp
}=\mathrm{id}$, we then get that $\pi_{1}(Y)$ is isomorphic to $\pi_{1}(X)$
via $v_{\sharp}=u_{\sharp}^{-1}$.

We now explain how, starting from a homotopy idempotent
$f\co X\rightarrow X$ of a finite connected complex $X$ with
fundamental group $G=\pi_{1}(X)$, we obtain an element $w(f)\in
K_{0}(\mathbb{Z}G)$. In the situation above, $Y$ is called
\emph{finitely dominated}; then the singular chain complex of the
universal cover $\tilde{Y}$ of $Y$ is chain homotopy equivalent to
a complex of type FP over $\mathbb{Z}\pi_{1}(Y)$
\[
0\rightarrow P_{n}\rightarrow\cdots\rightarrow P_{1}\rightarrow P_{0}%
\rightarrow\mathbb{Z}%
\]
with each $P_{i}$ a finitely generated projective $\mathbb{Z}\pi_{1}%
(Y)$--module. We then look at the \emph{Wall element}
\[
w(Y)=\sum_{i=0}^{n}(-1)^{i}[P_{i}]\in K_{0}(\mathbb{Z}\pi_{1}(Y))
\]
(where we follow the notation of Mislin \cite{Mislin-Handbook}). Its image
$$\tilde{w}(Y)=\eta(w(Y))\in\wtilde{K}_{0}(\mathbb{Z}\pi_{1}(Y))$$
is known as
Wall's \emph{finiteness obstruction}, and $\tilde{w}(Y)=0$ exactly when $Y$ is
homotopy equivalent to a finite complex. Finally, we define
\[
w(f)=v_{\sharp}(w(Y))\in K_{0}(\mathbb{Z}G),
\]
whose reduction $\tilde{w}(f)\in\tilde{K}_{0}(\mathbb{Z}G)$ was first
considered by Geoghegan \cite{Geo:LNM}. As he notes, the element $\tilde
{w}(f)$ ``can be interpreted as the obstruction to splitting $f$ through a
finite complex''. Before proceeding, we check that $w(f)$ is well-defined.
First, we observe a form of naturality of Wall elements.

\begin{Lemma}
\label{WallElement} Let $X$ be a finite $n$--dimensional complex,
and suppose that there are maps (of spaces having the homotopy
type of a connected
CW--complex)%
\[
X\overset{u}{\longrightarrow}W\overset{a}{\longrightarrow}V\overset
{v}{\longrightarrow}X
\]
such that $a\co W\rightarrow V$ and $u\circ v\co V\rightarrow W$ are homotopy
inverse. Then the Wall elements $w(W)\in K_{0}(\mathbb{Z}\pi_{1}(W))$ and
$w(V)\in K_{0}(\mathbb{Z}\pi_{1}(V))$ are related by%
\[
w(V)=a_{{\sharp}}w(W).
\]
\end{Lemma}

\begin{proof}
We use the fact that, because conjugation in $G$
induces the identity map on $K_{0}(\mathbb{Z}G)$, homotopic maps induce the
same homomorphism of $K$--groups. Recall from Wall \cite{Wall} that $w(W)$ is
defined (uniquely) by means of any $n$--connected map $\psi\co L\rightarrow W$
where $L$ is a finite $n$--dimensional complex:%
\[
w(W)=(-1)^{n}[\pi_{n+1}(M_{\psi},L)]
\]
where $M_{\psi}$ denotes the mapping cylinder of $\psi$, and the relative
homotopy group is considered as a $\pi_{1}(W)$--module (finitely generated and
projective because of the assumption that $W$ is dominated by $X$). Therefore,
to define $w(V)$, we may take
\[
w(V)=(-1)^{n}[\pi_{n+1}(M_{a\psi},L)]\text{.}%
\]
The result then follows from the natural isomorphism of the exact homotopy
sequences (of $\pi_{1}$--modules) of the pairs $(M_{\psi},L)$ and $(M_{a\psi
},L)$ induced by $a$.
\end{proof}

We now can see that the obstruction to splitting a homotopy
idempotent through a finite complex is well defined.
\begin{Lemma}
Let $X$ be a finite complex with fundamental group $G$, and for
$i=1,2$ let
$X\overset{u_{i}}{\longrightarrow}Y_{i}\overset{v_{i}}{\longrightarrow}X$
be a (homotopy) splitting of a homotopy idempotent map
$f\co X\rightarrow X$. Then in
$K_{0}(\mathbb{Z}G)$%
\[
v_{1{\sharp}}(w(Y_{1}))=v_{2{\sharp}}(w(Y_{2}))\text{.}%
\]
\end{Lemma}

\begin{proof}
From the homotopy commutative diagram
\[
\xymatrix{
& Y_1\ar[rr]^{\rm id}\ar[dr]_{v_1}& & Y_1\ar[dr]_{v_1}\ar[rr]^{\rm
id}& & Y_1\\
X\ar[rr]_f\ar[dr]_{u_2}\ar[ur]_{u_1}& & X\ar[rr]_f\ar[dr]_{u_2}\ar
[ur]_{u_1}& & X\ar[dr]_{u_2}\ar[ur]_{u_1} & \\
& Y_2\ar[rr]_{\rm id}\ar[ur]_{v_2}& & Y_2 \ar[rr]_{\rm
id}\ar[ur]_{v_2}& &Y_2
}%
\]
we deduce from a simple diagram chase that $a:=u_{2}\circ v_{1}\co Y_{1}%
\rightarrow Y_{2}$ and $b:=u_{1}\circ v_{2}\co Y_{2}\rightarrow Y_{1}$ are
mutually inverse homotopy equivalences. Therefore
\[
v_{2}\circ a\circ u_{1}\co X\rightarrow Y_{1}\rightarrow Y_{2}\rightarrow X
\]
is such that $a$ is homotopy inverse to $u_{1}\circ v_{2}\co Y_{2}\rightarrow
Y_{1}$; and thus, by \fullref{WallElement}, $a_{\sharp}(w(Y_{1}))=w(Y_{2})$.
It then follows that
\[
v_{1\sharp}(w(Y_{1}))=v_{1\sharp}\circ a_{\sharp}^{-1}(w(Y_{2}))=v_{1\sharp
}\circ u_{1\sharp}\circ v_{2\sharp}(w(Y_{2}))=f_{\sharp}\circ v_{2\sharp
}(w(Y_{2})).
\]
Similarly
\[
v_{2\sharp}(w(Y_{2}))=f_{\sharp}\circ v_{1\sharp}(w(Y_{1})).
\]
Then substituting in the previous formula, and using idempotency
of $f_{\sharp}$, gives the result.
\end{proof}

The key fact for giving a topological meaning to the Bass conjecture
is the following, which can be extracted from a result of Wall
\cite[Theorem~F]{Wall} in the light of the above. It is also shown
explicitly by Mislin \cite{OG}.

\begin{Theorem}
\label{all elts obstructions}Let $G$ be a finitely presented
group, let $\tilde{\alpha}\in\wtilde{K}_{0}(\mathbb{Z}G)$, and
let $n\geq2$. Then there is a finite $n$--dimensional complex
$X^{n}$ with fundamental group $G$ and a pointed homotopy
idempotent selfmap $f$ of $X^{n}$ inducing the
identity on $\pi_{1}$, such that $\tilde{w}(f)$ is equal to $\tilde{\alpha}%
$.
\end{Theorem}

\begin{rem}
\label{unreduced Wall}For $n\geq3$, the unreduced version of this result also
holds. For, given $\alpha\in K_{0}(\mathbb{Z}G)$, then choose a map $f$ as in
the theorem with respect to $\tilde{\alpha}=\eta(\alpha)$. It follows that for
some nonnegative $r,s$ we have $w(f)=\alpha+[\mathbb{Z}G]^{r}-[\mathbb{Z}%
G]^{s}$. Replacing $f$ by $f\vee\mathrm{id}_{W}$ where $W=\bigl(  \bigvee
_{r}S^{3}\bigr)  \vee\bigl(  \bigvee_{s}S^{2}\bigr)$ then gives the
desired selfmap. When $n=2$, the method fails, as without the possibility of
adjoining a simply-connected space of non-positive Euler characteristic we can
only increase the rank of the Wall element. (Recall from
Mislin \cite[Lemma~5.1]{Mislin-Handbook} that for any finitely dominated space
$Y$, the rank of $w(Y)$ equals $\chi(Y)$.)
\end{rem}

\section{Invariants for selfmaps of complexes\label{invar CW}}

We recall from, for example, the articles \cite{Geo:LNM,Geo:Hbk} by
Geoghegan, the
definition of the \emph{Nielsen number} $N(f)$ of a selfmap
$f\co X\rightarrow X$ of a finite connected complex (assumed, as
discussed above, to fix a basepoint of $X$). Let $f_{{\sharp}}$ be
the endomorphism of $G=\pi_{1}(X,x)$ induced by $f$.
Define elements $\alpha$ and $\beta$ of $G$ to be $f_{{\sharp}}$%
\emph{--conjugate} if for some $z\in G$%
\[
\alpha=z\cdot\beta\cdot(f_{{\sharp}}z)^{-1}\,\text{,}%
\]
and let $G_{f_{{\sharp}}}$ denote the set of $f_{{\sharp}}$--conjugacy classes,
making $\mathbb{Z}G_{f_{{\sharp}}}$ a quotient of $\mathbb{Z}G$. Now $N(f)$ is
defined to be the number of nonzero coefficients in the formula for the
\emph{Reidemeister trace} of $f$ at $x\in X$:%
\[
R(f,x)=\sum\nolimits_{C\in G_{f_{{\sharp}}}}n_{C}\cdot C\in\mathbb{Z}%
G_{f_{{\sharp}}}\text{.}%
\]
The coefficient $n_{C}$ can be described geometrically as the
fixed-point index of a fixed-point class of $f$, and homologically
in terms of traces of the homomorphisms induced by $f$ on the
chain complex of the universal cover of $X$. In the literature,
$R(f,x)$ is also known as the \emph{generalized Lefschetz number}.

When, as prompted by \fullref{all elts obstructions} above, we
take $f$ to induce the identity on $\pi_{1}$, then
$R(f,x)\in\bigoplus_{\lbrack s]\in\lbrack
G]}\mathbb{Z}\cdot\lbrack s]$ and the following holds
(cf Geoghegan \cite[p505]{Geo:Hbk}).

\begin{Lemma}[Geoghegan]
\label{GeogheganLemma} In the setting of
diagram \eqref{Heller} of \fullref{hi} where $X$ is a finite
connected complex, and $f$ is a pointed homotopy idempotent
selfmap inducing the identity map on the
fundamental group,%
\[
{\mathrm{HS}}({w}(f))=R(f,x).
\]
\end{Lemma}

For the computation of Nielsen numbers, the following result also proved
in Jiang \cite[p20]{Jiang-Lectures}, attributed to Fadell, is useful.

\begin{Lemma}
\label{Nielsens agree}Suppose that the diagram of finite connected
complexes and based maps
\[
\xymatrix{
\wwbar{T}\ar[rr]^{\bar{g}}\ar[d]^r & &\wwbar{T}\ar[drr]^r & & \\
T \ar[rr]^g & &T\ar[u]^s\ar[rr]^{\rm id} & &T
}
\]
 is commutative up to (free) homotopy. Then $N(g)=N(\bar{g})$.
\end{Lemma}

\begin{proof}
We use the definition and notation for $N(g)$
and $N(\bar{g})$ given above; we also put $\wwbar{G}=\pi_{1}(\wwbar{T}%
,\,s(t_{0}))$ where $t_{0}$ is the basepoint of $T$. Evidently, because
$s\circ g\simeq\bar{g}\circ s$, there is a well-defined function%
\[
s_{{\sharp}}\co G_{g_{{\sharp}}}\longrightarrow\wwbar{G}_{\bar{g}_{{\sharp}}}
\]
which, because also $r\circ\bar{g}\simeq g\circ r$, while $r\circ
s\simeq\mathrm{id}_{T}$, has left inverse $r_{{\sharp}}\co \wwbar{G}_{\bar
{g}_{{\sharp}}}\rightarrow G_{g_{{\sharp}}}$. In particular, $s_{{\sharp}}$ is
injective. With reference to the sentence before \cite[(2.2)]{Geo:LNM}, note
that this is not true in general without some condition on $s$, such as its
having a left inverse.

The formula for the Reidemeister trace of $g$ at $t_{0}\in T$ is:%
\[
R(g,t_{0})=\sum\nolimits_{C\in G_{g_{{\sharp}}}}n_{C}\cdot C\in\mathbb{Z}%
G_{g_{{\sharp}}}\text{.}%
\]
According to \cite[(2.2)]{Geo:LNM}, we also have%
\[
R(\bar{g},\,s(t_{0}))=\sum\nolimits_{C\in G_{g_{{\sharp}}}}n_{C}\cdot
s_{{\sharp}}(C)\in\mathbb{Z}\wwbar{G}_{\bar{g}_{{\sharp}}}\text{.}%
\]
Because $s_{{\sharp}}$ is injective, the two sums have the same number of
nonzero coefficients; that is, the Nielsen numbers agree.
\end{proof}

The next lemma permits us in our discussion of Nielsen numbers to restrict to
the case of those homotopy idempotents that are pointed homotopy idempotents
and induce the identity on the fundamental group.

\begin{Lemma}
{\label{reduction}} Suppose that $f\co X\rightarrow X$ is a homotopy
idempotent on a finite connected complex $X$, fixing $x_{0}\in X$,
with $G=\pi _{1}(X,x_{0})$ and $H:=f_{\sharp}(G)$. Then there is a
finite connected complex $K$ with fundamental group isomorphic to
$H$ and a pointed homotopy idempotent $g\co K\rightarrow K$, inducing
the identity map on $H$, such that $N(f)=N(g)$. Furthermore, if
$G$ satisfies the Bass conjecture, then so does $H$.
\end{Lemma}

\begin{proof}
Let
$X\overset{u}{\longrightarrow}Y\overset {v}{\longrightarrow}X$ be
a splitting for $f$. Then, by \cite[Corollary~5.5]{Mislin-Handbook},
$Y\times S^{3}$ is homotopy equivalent to a finite connected
complex $K$, because $Y$ is finitely dominated and the Euler
characteristic of $S^{3}$ is zero. Let $h\co Y\times S^{3}\rightarrow
K$ be a pointed homotopy equivalence, with pointed homotopy
inverse $k$; define $g\co K\rightarrow K$ to be the map
$h\circ(\mathrm{id}_{Y}\times\{\ast\})\circ k$, where
$\mathrm{id}_{Y}\times\{\ast\}\co Y\times S^{3}\rightarrow Y\times
S^{3}$ denotes the idempotent on $Y\times S^{3}$ given by the
projection onto $Y$. Clearly $g$ is a pointed homotopy idempotent,
inducing the identity on the fundamental group of $K$, and
$\pi_{1}(K)\cong\pi_{1}(Y)\cong H$.

Writing $u^{\prime}=u\times\mathrm{id}_{S^{3}}$ and
$v^{\prime}=v\times \mathrm{id}_{S^{3}}$, we now apply \fullref{Nielsens
agree} with $\wwbar{T}=X\times S^{3}$ and $T=K$,
and maps $\wbar{g}=f\times \{\ast\}$, $r=h\circ u^{\prime}$,
$s=v^{\prime}\circ k$ and $g$ as defined already. This yields
the following homotopy commutative diagram
\[%
\xymatrix{
X\times S^{3}\ar[rr]^{f\times\{\ast\}}\ar[d]^{u^{\prime}}\ar@(l,ul)[dd]_r & & X\times S^{3}\ar[dr]^{u'}\ar@(ur,u)[ddrr]^r & & \\
Y\times S^{3}\ar[d]^h &  & Y\times S^{3}\ar[u]^{v'} & Y\times S^{3}\ar[dr]^h &\\
K\ar[rr]^g &  & K\ar[rr]^{\mathrm{id}}\ar[u]^k\ar@(ul,l)[uu]^s & & K.
}
\]
We conclude that $N(g)=N(f\times\{\ast\})$.

On the other hand, $N(f\times\{\ast\})=N(f)$, as can be seen by again applying
\fullref{Nielsens agree}, with the top part of the diagram as before, but
$T=X$, $r=\mathrm{pr}_{X}$ and $s$ the inclusion $x\mapsto(x,\ast)$
\[%
\xymatrix{
X\times S^{3}\ar[rr]^{f\times\{\ast\}}\ar[d] &  & X\times S^{3}\ar[drr] & & \\
X\ar[rr]^f &  & X \ar[rr]^{\mathrm{id}}\ar[u]& & X.
}
\]
Therefore $N(f)=N(g)$. That $H\cong\pi_{1}(Y)$ satisfies the Bass conjecture
if $G$ does follows by observing that $v_{\sharp}\co \pi_{1}(Y)\rightarrow\pi
_{1}(X)$ is a split injection, and therefore the induced map $HH_{0}%
(\mathbb{Z}\pi_{1}(Y))\rightarrow HH_{0}(\mathbb{Z}\pi_{1}(X))$ is a split
injection too.
\end{proof}

We are now able to obtain a restatement of the theorem of
Geoghegan referred to in the Introduction (\cite[Theorem~4.1']{Geo:LNM}
(i)' $\Leftrightarrow$ (iii)'), in a form suitable to the
present treatment.

\begin{Theorem}[Geoghegan]
\label{Geoghegan Theorem 4.1'}
Let $G$ be a finitely presented group. The following are equivalent.
\item
\begin{itemize}
\item[\rm(a)]$G$ satisfies Bass' \fullref{Bass}.
\item[\rm(b)] Every homotopy idempotent selfmap $f$ on a finite
connected complex with fundamental group $G$ has Nielsen number
either zero or one.
\end{itemize}
\end{Theorem}

\begin{proof}
We start with the implication $\mathrm{(a)} \Rightarrow \mathrm{(b)}$. Let $f$ be as in (b). Then by \fullref{reduction} we can assume
that $f\co X\rightarrow X$ is actually a pointed homotopy idempotent
on a finite connected complex $X$ that induces the identity map on
$\pi_{1}(X,x_{0})\cong G$. Because $G$
satisfies the Bass conjecture, we have $\mathrm{HS}(w(f))\in$ $\mathbb{Z}%
\cdot\lbrack e]$. Then, by \fullref{GeogheganLemma}, $R(f,x_{0})$ has at
most one nonzero coefficient, and $N(f)\leq1$.

In the other direction, we of course use \fullref{all elts obstructions}
and \fullref{unreduced Wall}. Then, given ${\alpha}\in{K}_{0}(\mathbb{Z}%
G)$, there is a finite $n$--dimensional complex $X$ ($n\geq3$) with
fundamental group $G$ and a pointed homotopy idempotent selfmap
$f$ of $X$ inducing the identity on $\pi_{1}(X)=G$, such that ${w}(f)$
is equal to ${\alpha}$. So, from \fullref{GeogheganLemma} we
deduce that ${\mathrm{HS}}({\alpha })=R(f,x)$. This last term
vanishes when $N(f)=0$; if $N(f)=1$, $R(f,x)$ is a nonzero
multiple of some class $[s]$, and we are done in case $[s]=[e]$.
So, the remaining case is where $R(f,x)$ is a nonzero multiple
(necessarily $\chi(Y)$) of some class $[s]\neq\lbrack e]$. In that
event we may turn instead to $f^{\prime}=f\vee\mathrm{id}_{S^{2}}$
with corresponding $Y^{\prime}=Y\vee S^{2}$ having $w(f^{\prime})$
$=w(f)+[\mathbb{Z}G]$. However, this implies the contradiction
that $N(f^{\prime})=2$, and can therefore be
eliminated.
\end{proof}

\begin{rem}
The actual wording of \cite[Theorem~4.1']{Geo:LNM} is in terms of
the Bass conjecture for a particular element $\alpha$ of
$K_{0}(\mathbb{Z}G)$, rather than for $G$ itself.
\end{rem}

\section{Selfmaps of manifolds\label{CW->M}}

From Wecken's work \cite{W}, one knows that $N(f)$ serves as a lower bound for
the number of fixed points of any map homotopic to $f$. Thus, the implication
(ii) $\Rightarrow$ (i) in the next result is immediate.

\begin{Lemma}
\label{Nielsen to unique fp}Suppose that $f\co M\rightarrow M$ is a selfmap of a
closed manifold $M$ of dimension at least $3$. Then the following are equivalent:

\begin{enumerate}
\item[\rm(i)]the Nielsen number of $f$ is $0$ or $1$;

\item[\rm(ii)] $f$ is homotopic to a map having one arbitrarily chosen unique
fixed point.
\end{enumerate}
\end{Lemma}

\begin{proof}
We need only prove that (i) implies (ii).

First, various results in the literature (see in particular Wecken
\cite{W}, Brown \cite{Brown}, and Shi \cite{Shi} for the PL case; Jiang
\cite{Jiang-LNM886} for the smooth case) show that every selfmap of $M$
with Nielsen number $N$ is homotopic to a map with exactly $N$ fixed
points. By a result of Schirmer \cite[Lemma~2]{Schirmer}, these fixed
points may be chosen arbitrarily.

Second, recall from \cite{Schirmer} that every fixed-point-free selfmap of a
connected compact PL manifold of dimension at least $3$ is homotopic to a
selfmap having an arbitrary unique fixed point. The argument there can be
adapted as follows.

Choose $a\in M$, and consider the closure $\bar{B}$ of an open ball $B$ around
$a$ lying in a chart domain for $M$. Since $f$ is fixed-point-free, we choose
the ball $\bar{B}$ to be small enough so that it is disjoint from its image
under $f$. For convenience, we consider points of $\bar{B}$ with coordinates
so that $a=\mathbf{0}$ and $\bar{B}$ consists of those points $x$ with
$\left\|  x\right\|  \leq1$. Now let $\gamma$ be any path from $a$ to $f(a)$
that continues a unit-speed ray from $a$ to the boundary of $\bar{B}_{1/2}$
(the closed ball of points of $B$ of norm at most $1/2$) and never re-enters
$\bar{B}_{1/2}$. Also, let $\lambda\co M\rightarrow\lbrack0,1]$ be a function
having $\lambda^{-1}(1)=M-B$ and $\lambda^{-1}(0)=\bar{B}_{1/2}$. Then the
desired map $g\co M\rightarrow M$ homotopic to $f$ is given as follows.%
\[
g(x)=\left\{
\begin{array}
[c]{lll}%
\gamma(t_{x}) & \quad & 0\leq2\left\|  x\right\|  <1\\
f((1-\lambda(x))a+\lambda(x)x) &  & 1\leq2\left\|  x\right\|  \leq2\\
f(x) &  & x\in M-\bar{B}\ \text{,}%
\end{array}
\right.
\]
where $t_{x}=1-\exp(-2\left\|  x\right\|  /(1-2\left\|  x\right\|  ))$. Here,
recall the standard inequality
\[
\ln(1-u)>-u/(1-u)
\]
for $0<u<1$. It implies that, whenever $t_{x}\neq0$ and $\gamma(t_{x})\in
\bar{B}_{1/2}\,$, so that $\left\|  \gamma(t_{x})\right\|  =t_{x}\,$,
\[
\left\|  \gamma(t_{x})\right\|  >2\left\|  x\right\|  \,\text{.}%
\]
Hence $a$ is the unique fixed point of $g$.

Note that it is possible to make $g$ smooth. For, since every map is homotopic
to a smooth map and homotopy does not change Nielsen numbers, there is no loss
of generality in assuming $f$ to be smooth. Then, by our taking both $\gamma$
and $\lambda$ to be smooth functions in the above argument, a smooth map $g$
results.
\end{proof}

\begin{rem}
Note that this result cannot be extended to dimension $2$ in
general. Indeed, for every connected, closed surface of negative
Euler characteristic and every natural number $n$, Jiang \cite[Theorem~2]{Jiang:non-Wecken} exhibits a selfmap $f_{n}$ of the surface
having $N(f_{n})=1$, but with every map homotopic to $f_{n}$
having more than $n$ fixed points. For some particular results on
selfmaps on surfaces, see also Kelly \cite{K}.
\end{rem}

We next observe that selfmaps of complexes may be studied by means of selfmaps
of manifolds without changing the Nielsen number.

\begin{Lemma}
\label{mani} \label{complex to manifold}Let $X$ be a finite
connected complex. Then the following hold.

\begin{enumerate}
\item[\rm(a)]There is a closed, oriented and smooth manifold $M$ of dimension at
least $3$ with maps $r\co M\rightarrow X$ and $s\co X\rightarrow M$ having $r\circ
s$ pointed homotopic to $\mathrm{id}_{X}$ and inducing isomorphisms of
fundamental groups.

\item[\rm(b)] For any selfmap $f\co X\rightarrow X$, the selfmap $\bar{f}=s\circ
f\circ r\co M\rightarrow M$ has Nielsen number%
\[
N(\bar{f})=N(f)\text{.}%
\]

\item[\rm(c)] If $f$ is either homotopy idempotent or pointed homotopy
idempotent, then so is $\bar{f}$.
\end{enumerate}
\end{Lemma}

\begin{proof}
(a)\qua Working up to pointed homotopy type, we may assume
that $X$ is a finite simplicial complex of dimension $n\geq2$. By a
result of Wall \cite[Theorem~1.4]{Wall:AnnM1966} we can do surgery on the
constant map $S^{2n}\rightarrow X$ to obtain a smooth, oriented (indeed,
stably parallelizable) closed $2n$--manifold $M$ and an $n$--connected map
(called an $n$--equivalence by Spanier \cite{Spanier}) $r\co M\rightarrow
X$. Because $n\geq2$, the map $r$ is a $\pi_{1}$--isomorphism. Moreover,
since the obstruction groups $H^{i}(Y;\pi_{i}(r))$ all vanish (or by
\cite[(7.6.13)]{Spanier}), the identity map $X\rightarrow X$ factors up
to pointed homotopy through $M\rightarrow X$, and the result follows.

(b)\qua This result is immediate from \fullref{Nielsens agree} above,
on putting $\wwbar{T}=M,\ T=X,\ g=f$ and $\bar{g}=\bar{f}$.

(c)\qua Obviously,%
\begin{align*}
\bar{f}\circ\bar{f}  &  \simeq s\circ f\circ r\circ s\circ f\circ r\\
&  \simeq s\circ f\circ f\circ r\simeq s\circ f\circ r\simeq\bar{f}\text{,}%
\end{align*}
and $\bar{f}$ is a pointed idempotent if $f$ is.
\end{proof}

\begin{exm}
For any connected non-contractible space $X$, the monoid of homotopy classes
of selfmaps of $X$ always contains at least two idempotents, the class of
nullhomotopic maps and the class of maps homotopic to the identity. Each
constant map in the former class contains exactly one fixed point (which by
connectivity is arbitrary), and obviously has Nielsen number $1$.

On the other hand, for $X$ a finite complex the identity map has Nielsen
number equal to $\min\{1$, $\left|  \chi(X)\right|  \}$. When $X$ is also a
smooth manifold, it admits a smooth vector field whose only singularity is an
arbitrarily chosen point $x_{0}\in X$. Its associated flow provides a homotopy
from the identity map to a smooth map with sole fixed point $x_{0}$.
\end{exm}

\section[Proof of \ref{principal} and applications]
  {Proof of \fullref{principal} and applications}
\label{prfThm1}

The discussion above now allows a reformulation of statement (b)
of \fullref{Geoghegan Theorem 4.1'}. \fullref{complex to manifold}
combines with \fullref{Nielsen to unique fp} to yield a manifold version
of statement (b), as follows.

\begin{Proposition}
\label{PrincipalUtile}Let $G$ be a finitely presented group. The following are equivalent.

\begin{itemize}
\item[\rm(a)]$G$ satisfies Bass' \fullref{Bass}.

\item[\rm(b)] Given any closed, smooth and oriented manifold $M$ of dimension at
least $3$ with $G=\pi_{1}(M)$, every homotopy idempotent selfmap $f$ on $M$ is
homotopic to one that has a single fixed point.
\end{itemize}
\end{Proposition}

\begin{rem}
The following facts combine to show that the dimension $3$ in (b)
above is best possible. For $F$ a closed surface of negative Euler
characteristic and $n \ge 2$, Kelly \cite{Kelly-new} constructs a
homotopy idempotent selfmap $f_n\co F\rightarrow F$ such that every
map homotopic to $f_n$ has at least $n$ fixed points. On the other
hand, the fundamental groups of surfaces are well-known to satisfy
Bass' \fullref{Bass} (see Eckmann \cite{EId}, for example).
\end{rem}

The following argument of Bass, reported by R Geoghegan, shows
that it suffices to consider finitely presented groups in
considering Bass' conjectures.

\begin{Lemma}[Bass]
\label{GeoBass}
Conjectures \ref{Bass} and \ref{WeakBass} hold for all groups if they
hold for all finitely presented groups.
\end{Lemma}

\begin{proof}
Fix a group $G$. We show that any idempotent
$\mathbb{Z}G$--matrix $A$ lifts to an idempotent matrix $A_{1}$ over the group
ring of a finitely presented group $G_{1}$. There is a finitely generated
subgroup $G_{0}$ of $G$ such that the entries of $A$ lie in $\mathbb{Z}G_{0}$.
Write $G_{0}$ as $F/R$ where $F$ is a finitely generated free group; and let
$B$ be a lift of $A$ to $\mathbb{Z}F$. Then there is a finite subset $W$ of
$R$ such that the matrix $B^{2}-B$ has all its entries in the ideal of
$\mathbb{Z}F$ generated by $\{1-r\mid r\in W\}$. Now let $R_{1}\leq R$ be the
normal closure of $W$ in $F$. Then we have $G_{1}:=F/R_{1}$ finitely
presented, and the image $A_{1}$ of $B$, with entries in $\mathbb{Z}G_{1}$, is
an idempotent matrix. The map $G_{1}\twoheadrightarrow G_{0}\hookrightarrow G$
takes $A_{1}$ to $A$. Therefore $[A_{1}]\mapsto\lbrack A]$ under the induced
map $K_{0}(\mathbb{Z}G_{1})\rightarrow K_{0}(\mathbb{Z}G)$. Then the result
follows from naturality of $\mathrm{HS}$ .
\end{proof}

After \fullref{PrincipalUtile}, it is now straightforward to deduce
\fullref{principal}.\hfill\qed

\begin{rem}
Our arguments lead to variations on Theorem 1. First, one can sharpen the
implication (b) $\Rightarrow$ (a) by referring in (b) to a smaller class of
manifolds. Because, by \fullref{all elts obstructions} above, for a
finitely presented group $G$ any $\tilde{\alpha}\in\tilde{K}_{0}(\mathbb{Z}G)$
can be realized by a homotopy idempotent selfmap $f$ of a $2$--dimensional
complex with fundamental group $G$ (so $\tilde{w}(f)=\tilde{\alpha}$), the
Bass conjecture is equivalent to the following: \emph{Every homotopy
idempotent selfmap of a closed, stably parallelizable smooth $4$--manifold is
homotopic to one with a single fixed point.}

In the other direction, one can strengthen (a) $\Rightarrow$ (b) by enlarging
the class of spaces to which (b) applies. There is no need to restrict
attention to oriented, smooth manifolds; one can also apply to PL manifolds
and other, possibly bounded, Wecken spaces (see Jiang
\cite{Jiang-AmJM1980}).
\end{rem}

As an application of \fullref{PrincipalUtile} we obtain the following.

\begin{Corollary}
Any homotopy idempotent selfmap on a closed, smooth and oriented
$3$--dimensional manifold $M$ is homotopic to one with a single fixed point.
\end{Corollary}

\begin{proof}
It is enough to show that the fundamental group
$G$ of a closed smooth oriented $3$--dimensional manifold $M$ satisfies Bass'
conjecture; the Corollary then follows from \fullref{PrincipalUtile}.
By Kneser's result (see Milnor \cite{Milnor}), $M$ is a connected sum of prime
manifolds $M_{i}$, where each $M_{i}$ belongs to one of the following classes:

\begin{enumerate}
\item $M_{i}$ with finite fundamental group;

\item $M_{i}$ with fundamental group $\mathbb{Z}$;

\item $M_{i}$ a $K(\pi,1)$ manifold (so $\pi$ is a Poincar\'{e} duality group).
\end{enumerate}

Note that the fundamental group of $M$ is the free product of the fundamental
groups of the various $M_{i}$. By Gersten's result \cite{G}, given two groups
$\Gamma$ and $H$, the reduced projective class group of the free product
$\Gamma\ast H$ reads
\[
\tilde{K}_{0}(\mathbb{Z}(\Gamma\ast H))\cong\tilde{K}_{0}(\mathbb{Z}%
\Gamma)\oplus\tilde{K}_{0}(\mathbb{Z}H)\;.
\]
Thus, every element in $\tilde{K}_{0}(\mathbb{Z}(\Gamma\ast H))$ is an
integral linear combination of projectives induced up from $\Gamma$ and $H$
respectively. It follows that if Bass' conjecture holds for both $\Gamma$ and
$H$, then it holds for $\Gamma\ast H$ as well.

In the list above, clearly finite groups and $\mathbb{Z}$ satisfy Bass'
conjecture. That 3--di\-men\-sion\-al Poincar\'e duality groups satisfy Bass' conjecture follows from Eckmann's
work (see Eckmann \cite[p247]{EId}) on groups of rational cohomological dimension
$2$.
\end{proof}

Since the Bass conjecture is known for instance for the fundamental groups of
manifolds in the class below \cite{EId}, we have another consequence.

\begin{Corollary}
Any homotopy idempotent selfmap of a non-positively curved, oriented closed
manifold of dimension at least $3$ is homotopic to a map with a single fixed
point.
\end{Corollary}

It would be interesting to see geometric proofs of these facts.

\section{Lefschetz numbers}

\label{Lef}

Let $X$ be a CW--complex and $f\co X\rightarrow X$ a continuous selfmap. Then $f$
induces for each $n\in\mathbb{N}$ a map
\[
f_{n}\co H_{n}(X;\,\mathbb{Q})\rightarrow H_{n}(X;\,\mathbb{Q})
\]
of $\mathbb{Q}$--vector spaces. If the sum of the dimensions of the vector
spaces $H_{n}(X;\,\mathbb{Q})$ is finite, the \emph{Lefschetz number }of\emph{
}$f$ is defined as
\[
L(f)=\sum_{n\geq0}(-1)^{n}{\mathrm{{Tr}}}(f_{n}).
\]
In cases where the CW--complex $X$ is finite or finitely dominated, the
Lefschetz number of a selfmap is obviously always defined, and for
$f=\mathrm{id}_{X}$, $L(f)=\chi(X)$ the Euler characteristic of $X$. One
extends this definition to the case of $G$--CW--complexes as follows.

Let $\mathcal{N}\!G$ denote the von Neumann algebra of the discrete group $G$
(\textsl{i.e.} the double commutant of $\mathbb{C}G$ considered as a
subalgebra of the algebra of bounded operators on the Hilbert space $\ell
^{2}G$ -- see for example L\"uck \cite{Lueck}). With $e$ as the neutral element of
$G$, write $e\in G\subset\ell^{2}G$ for the delta-function
\[
{e}\co G\rightarrow\mathbb{C},\quad g\mapsto\left\{
\begin{array}
[c]{cc}%
1 & \hbox{ if }g=e\\
0 & \hbox{ otherwise.}%
\end{array}
\right.
\]
The standard trace
\[
\mathrm{{tr}}_{G}\co \mathcal{N}\!G\rightarrow\mathbb{C},\quad
x\mapsto\langle xe,e\rangle_{\ell^{2}G}\in\mathbb{C}%
\]
extends to a trace $\mathrm{{tr}}_{G}(\phi)\in\mathbb{C}$ for $\phi
\co M\rightarrow M$ a map of finitely presented $\mathcal{N}\!G$--modules as
follows. Recall that a finitely presented $\mathcal{N}\!G$--module $M$ is of
the form $S\oplus T$ with $S$ projective and $T$ of von Neumann dimension $0$;
the trace of $\phi$ is then defined as the usual von Neumann trace of the
composite $S\rightarrow M\rightarrow M\rightarrow S$ (for the trace of
selfmaps of finitely generated projective $\mathcal{N}\!G$--modules see
L\"uck \cite{Lueck}). It follows that $\mathrm{{tr}}_{G}(\mathrm{id}_{M})=\dim
_{G}(M)$, the von Neumann dimension of the finitely presented $\mathcal{N}%
\!G$--module $M$ (a non-negative real number cf \cite{Lueck}). The
Kaplansky trace (as defined in \fullref{Rewiew}) induces a trace on
$G$--maps $\psi\co P\rightarrow P$ of finitely generated projective $\mathbb{Z}%
G$--modules, and
\[
{\kappa}(\psi)=\mathrm{{tr}}_{G}(\mathrm{id}_{\mathcal{N}\!G}\otimes\psi)
\]
where $\mathrm{id}_{\mathcal{N}\!G}\otimes\psi\co \mathcal{N}\!G\otimes
_{\mathbb{Z}G}P\rightarrow\mathcal{N}\!G\otimes_{\mathbb{Z}G}P$.

Now let $Z$ be a free $G$--CW--complex that is dominated by a
cocompact $G$--CW--complex (for example, the universal cover of a
finitely dominated CW--complex with fundamental group $G$). A
$G$--map $\tilde{f}\co Z\rightarrow Z$ induces a map of singular chain
complexes $C_{\ast}(Z)\rightarrow C_{\ast}(Z)$ and of
$L^{2}$--chain complexes
\[
C_{\ast}^{(2)}(Z):=\mathcal{N}\!G\otimes_{\mathbb{Z}G}C_{\ast}(Z)\rightarrow
C_{\ast}^{(2)}(Z)\text{,}%
\]
and therefore of $L^{2}$--homology groups
\[
H_{n}^{(2)}(Z):=H_{n}(Z;\,\mathcal{N}\!G)\rightarrow H_{n}^{(2)}(Z).
\]
The groups $H_{n}^{(2)}(Z)$ are finitely presented $\mathcal{N}\!G$--modules,
because the complex $C_{\ast}^{(2)}\!(Z)$ is chain homotopy equivalent to a
complex of type FP over $\mathcal{N}\!G$ and because the category of finitely
presented $\mathcal{N}\!G$--modules is known to be abelian \cite{Lueck}. Thus the induced map
$$\wtilde{f}_n\co H_n^{(2)}(Z)\to H_n^{(2)}(Z)$$
is a selfmap of a finitely presented $\mathcal{N}\!G$--module and has, therefore, a well defined trace $\mathrm{{tr}}_G(\wtilde{f}_n)$ as explained in the beginning of this section; we also write $\beta_n^{(2)}(Z;G)$ for the von Neumann dimension of $H_n^{(2)}(Z)$.

Let $Y$ be a finitely dominated CW--complex with fundamental group $G$ and
universal cover $Z$. Then the $n$\thinspace th $L^{2}$--Betti number $\beta
_{n}^{(2)}(Y)$ of $Y$ is defined to be $\beta_{n}^{(2)}(Z;G)$. (If $Y$ happens
to be a finite complex, this reduces to the usual $L^{2}$--Betti number of $Y$
as defined for instance in Atiyah \cite{atiyah} and Eckmann \cite{EM}.) By definition, the
alternating sum $\sum(-1)^{i}\beta_{i}^{(2)}(Y)=\chi^{(2)}(Y)$ is the $L^{2}%
$--Euler characteristic of $Y$. Recall that for $Y$ a finite complex,
$\chi(Y)=\chi^{(2)}(Y)$ by Atiyah's formula \cite{atiyah}; see also
Chatterji--Mislin \cite{us}
and \fullref{wBass=Euler} below for more general results. We now define
$L^{2}$--Lefschetz numbers as follows.

\begin{Definition}
Let $Z$ be be a free $G$--CW--complex that is dominated by a cocompact
$G$--CW--complex and let $\wtilde{f}\co Z\rightarrow Z$ be a $G$--map.
Denote by ${\wtilde{f}}_n\co H_{n}^{(2)}(Z)\rightarrow H_{n}^{(2)}(Z)$ the induced map in $L^{2}%
$--homology. Then the $L^{2}$\emph{--Lefschetz number }of $\wtilde{f}$ is given by
\[
L^{(2)}(\wtilde{f}):=\sum_{n\geq0}(-1)^{n}\mathrm{tr}_{G}(\wtilde{f}_n)\in
\mathbb{R}.
\]
\end{Definition}

In case $Z$ is cocompact our $L^{2}$--Lefschetz number agrees with the one
defined by L\"uck and Rosenberg \cite[Remark~1.7]{LR}. If $Y$ is a
finitely dominated connected CW--complex with fundamental group $G$, and with
universal cover the free $G$--space $\tilde{Y}$, then the $L^{2}$--Lefschetz
number of the identity map of $\tilde{Y}$ is $\chi^{(2)}(\tilde{Y}%
;\,G)=\chi^{(2)}(Y)$, the $L^{2}$--Euler characteristic of $Y$.

\section[Proof of \ref{weakelyprincipal}]
  {Proof of \fullref{weakelyprincipal}}\label{Proof 2}

Let $Y$ be a finitely dominated connected CW--complex. Thus $\chi(Y)$ and
$\chi^{(2)}(Y)$ are defined as above, and are related as follows.

\begin{Lemma}
\label{wBass=Euler}Let $G$ be a finitely presented group. Then the following holds.

\begin{itemize}
\item[\rm(a)]Let $Y$ be a finitely dominated connected CW--complex with
fundamental group $G$. If the finiteness obstruction $\tilde{w}(Y)\in\tilde
{K}_{0}(\mathbb{Z}G)$ is a torsion element, then $\chi^{(2)}(Y)=\chi(Y)$.

\item[\rm(b)] The following are equivalent.

\begin{enumerate}
\item[\rm(i)]The weak Bass conjecture holds for $G$.

\item[\rm(ii)] For any finitely dominated connected CW--complex $Y$ with $\pi
_{1}(Y)=G$, we have
\[
\chi^{(2)}(Y)=\chi(Y).
\]
\end{enumerate}
\end{itemize}
\end{Lemma}

\begin{proof}
As in \fullref{hi}, for $Y$ finitely
dominated, the chain complex $C_{\ast}(\tilde{Y})$ is chain homotopy
equivalent to a chain complex $P_{\ast}$ of type FP over $\mathbb{Z}G$,
$G=\pi_{1}(Y)$, and we have the Wall element
\[
w(Y)=\sum_{i=0}^{n}(-1)^{i}[P_{i}]\in K_{0}(\mathbb{Z}G).
\]
As $\mathcal{N}\!G\otimes_{\mathbb{Z}G}P_{\ast}\simeq C^{(2)}(\tilde{Y})$ and
$\mathrm{{tr}}_{G}(\mathcal{N}\!G\otimes_{\mathbb{Z}G}P_{i})=\kappa(P_{i})$,
we see that
\[
\chi^{(2)}(Y)=\sum(-1)^{i}\mathrm{{tr}}_{G}(\mathcal{N}\!G\otimes
_{\mathbb{Z}G}P_{i})=\sum(-1)^{i}\kappa(P_{i})=\kappa(w(Y)).
\]
On the other hand, $\mathbb{Z}\otimes_{\mathbb{Z}G}P_{\ast}\simeq C_{\ast}(Y)$
and $\epsilon(P_{i})=\mathrm{dim}_{\mathbb{Q}}(\mathbb{Q}\otimes_{\mathbb{Z}%
G}P_{i})$ so that
\[
\chi(Y)=\sum(-1)^{i}\mathrm{dim}_{\mathbb{Q}}(\mathbb{Q}\otimes_{\mathbb{Z}%
G}P_{i})=\sum(-1)^{i}\epsilon(P_{i})=\epsilon(w(Y)).
\]

(a)\qua Again we observe that for $n>1$,
\[
\chi(Y\vee(\vee^{k}S^{n}))=\chi(Y)+(-1)^{n}k
\]
and
\[
\chi^{(2)}(Y\vee(\vee^{k}S^{n}))=\chi^{(2)}(Y)+(-1)^{n}k.
\]
On the other hand, $w(Y\vee(\vee^{k}S^{n}))=w(Y)+(-1)^{n}k$, so that without
loss of generality we may assume that actually $w(Y)$ is a torsion element.
But then, since the range of the Hattori--Stallings trace is torsion-free,
Bass' conjectures are valid for torsion elements of $K_{0}(\mathbb{Z}G)$ and
we have%
\[
\kappa(w(Y))=\epsilon(w(Y))\text{.}%
\]
Thus,
\[
\chi^{(2)}(Y)=\chi(Y),
\]
proving the claim.

(b)\qua (i) $\Rightarrow$ (ii): Assuming the weak Bass
conjecture, we have
\[
\chi^{(2)}(Y)=\kappa(w(Y))=\epsilon(w(Y))=\chi(Y).
\]
Assuming the Bass conjecture, this implication has also been proved by Eckmann
\cite{EN}.

(ii) $\Rightarrow$ (i): Recall from \fullref{all elts obstructions}
that for a finitely generated projective $\mathbb{Z}G$--module
$P$ there is always a finitely dominated CW--complex $Y$ whose Wall element
$w(Y)$ equals $[P]\in K_{0}(\mathbb{Z}G)$. We then have that
\[
\kappa(P)=\kappa(w(Y))=\chi^{(2)}(Y)=\chi(Y)=\epsilon(w(Y))=\epsilon(P).
\proved
\]
\end{proof}

Next, we need an intermediate result.

\begin{Lemma}
\label{Lef=Euler}Let $X$ be a finite connected complex, and
$f\co X\rightarrow X$ be a homotopy idempotent. Let $Y$ be a finitely
dominated CW--complex determined by $f$ as in \fullref{hi}.

\begin{enumerate}
\item[\rm(a)]Then $L(f)=\chi(Y)$.

\item[\rm(b)] If moreover, $f$ is a pointed homotopy idempotent inducing the
identity on $G=\pi_1(X)$ and $\wtilde{f}$ denotes the induced $G$--map on the universal cover of $X$, then $L^{(2)}(\wtilde{f})=\chi^{(2)}(Y)$.
\end{enumerate}
\end{Lemma}

\begin{proof} (a)\qua Applying $H_{i}(\,-\,;\,\mathbb{Q})$ to the
diagram \eqref{Heller} of \fullref{hi} yields the following commutative
diagram of groups:
\[
\xymatrix{
H_i(X)\ar[rr]^{f_i}\ar[dr]_{u_{i}}\ar@(ur,ul)[rrrr]^{f_i}& &H_i(X)\ar[rr]^{f_{i}}
\ar[dr]_{u_{i}}&  &H_i(X)\\
&H_i(Y)\ar[rr]^{{\mathrm{id}}_i}\ar[ur]_{v_{i}} & &H_i(Y)\ar[ur]_{v_{i}}.&
}%
\]
We can now compute that
\[
\mathrm{{Tr}}(f_{i})=\mathrm{{Tr}}(v_{i}u_{i})=\mathrm{{Tr}}%
(u_{i}v_{i})=\mathrm{{Tr}}(\mathrm{id}_{i})=\dim(H_{i}(Y;\,\mathbb{Q}%
))\text{,}%
\]
so that (a) follows by taking alternating sums.

(b)\qua Here we need to know that $f$ induces the identity map
on $\pi_{1}(X)$, in order to obtain equivariance of the induced maps on the
universal covers $\tilde{X}$ and $\tilde{Y}$. We apply $H_{i}^{(2)}$ to
diagram \eqref{Heller}:
\[
\xymatrix{
\Hei(\tilde{X})\ar[rr]^{\wtilde{f}_i}\ar[dr]_{\wtilde{u}_{i}}\ar@(ur,ul)[rrrr]^{\wtilde{f}_i}&
&\Hei(\tilde{X})\ar[rr]^{\wtilde{f}_{i}}\ar[dr]_{\wtilde{u}_{i}%
}&  &\Hei(\tilde{X})\\
&\Hei(\tilde{Y})\ar[rr]^{{\mathrm{id}}_i}\ar[ur]_{\wtilde{v}_{i}}%
& &\Hei(\tilde{Y})
\ar[ur]_{\wtilde{v}_{i}},& }%
\]
and compute
\[
\mathrm{{tr}}_{G}(\wtilde{f}_{i})=\mathrm{{tr}}_{G}(\wtilde{v}%
_{i}\wtilde{u}_{i})=\mathrm{{tr}}_{G}(\wtilde{u}_{i}\wtilde
{v}_{i})=\mathrm{{tr}}_{G}(\mathrm{id}_{i})=\dim_{G}(H_{i}^{(2)}(\tilde
{Y}))\text{,}%
\]
and take alternating sums. The desired equality uses the fact that, given two
finitely presented $\mathcal{N}\!G$--modules $A$ and $B$, with two maps
$\phi\co A\rightarrow B$ and $\psi\co B\rightarrow A$, then $\mathrm{{tr}}_{G}%
(\phi\psi)=\mathrm{{tr}}_{G}(\psi\phi)$.
\end{proof}

\begin{Proposition}
Let $G$ be a finitely presented group. The following are equivalent.

\begin{itemize}
\item[\rm(a)]$G$ satisfies the weak Bass conjecture.

\item[\rm(b)] Every pointed homotopy idempotent selfmap of a closed, smooth and
oriented manifold $M$ with $\pi_{1}(M)=G$ and inducing the identity on $G$ has
its Lefschetz number equal to the $L^{2}$--Lefschetz number of the induced
$G$--map on the universal cover
of $M$.
\end{itemize}
\end{Proposition}

\begin{proof}
That (a) implies (b) follows from the
implication (a) $\Rightarrow$ (b) in \fullref{wBass=Euler}, combined
with \fullref{Lef=Euler}. To prove that (b) implies (a), namely that the
Lefschetz number information on manifolds is enough to imply the weak Bass
conjecture, it suffices to see that for a finite connected complex
$X$ of dimension $n\geq2$ there are a closed smooth oriented manifold
$M$ and maps $X\rightarrow M$, $M\rightarrow X$ inducing isomorphisms
of the fundamental groups, and such that $X\rightarrow M\rightarrow X$
is pointed homotopic to $\mathrm{id}_{X}$. However, this was already
discussed in \fullref{mani}. We then conclude by combining the implication
(b) $\Rightarrow$ (a) in \fullref{wBass=Euler} with \fullref{Lef=Euler}.
\end{proof}

Finally, we turn to the proof of \fullref{weakelyprincipal}. That (a)
implies (b) follows from the previous proposition, which also shows that (b)
implies (a) for all finitely presented groups, and therefore for all groups
via \fullref{GeoBass}. \hfill\qed

\begin{rem}
For each group $G$, it is evident that the algebraic statement of
\fullref{Bass}
for $G$ implies the statement of \fullref{WeakBass}. For our geometric formulations
of the conjectures, the implication is less clear. One can approach this
problem via work of L\"{u}ck and Rosenberg on computing $L^{2}$--Lefschetz
numbers and local degrees \cite{LR}.
\end{rem}

\bibliographystyle{gtart}
\bibliography{link}

\end{document}